\documentclass[10pt, article]{amsart}
\usepackage{tikz}
\usetikzlibrary{calc}
\usepackage{ae} % or {zefonts}
\usepackage[T1]{fontenc}
\usepackage[cp1250]{inputenc}
\usepackage{amsmath}
\usepackage{amssymb, amsfonts,amscd,verbatim}

\usepackage[normalem]{ulem}
\usepackage{hyperref}
\usepackage{indentfirst}
\usepackage{latexsym}
\input xy
\xyoption{all}

\usepackage{xcolor}

\usepackage{amsmath}    % need for subequations
\theoremstyle{plain}
\newtheorem{Pocz}{Poczatek}[section]
\newtheorem{Proposition}[Pocz]{Proposition}

\newtheorem{Theorem}[Pocz]{Theorem}
\newtheorem{Corollary}[Pocz]{Corollary}

\newtheorem{Observation}[Pocz]{Observation}

\newtheorem{Example}[Pocz]{Example}

\theoremstyle{definition}
\newtheorem{Definition}[Pocz]{Definition}

\theoremstyle{remark}

\newtheorem{Exercise}[Pocz]{Exercise}

\numberwithin{equation}{section}
\title[Unifying Linear Algebra]
{Unifying Linear Algebra}

\author{Jerzy Dydak}
\address{University of Tennessee, Knoxville, TN 37996, USA}
\email{jdydak@utk.edu}

\date{ \today
}


\begin{document}
\maketitle
\begin{center}
\today
\end{center}

\tableofcontents

\begin{abstract}
We unify Linear Algebra by proposing a definition of determinants via one equation that implies all known properties of them:\\
1. Cramer's Rule,\\
2. Cofactor expansion,\\
3. Antisymmetry of determinants,\\
4. Linearity of determinants,\\
5. Uniqueness of determinants up to a constant.\\
6. $\det(A\cdot B)=\det(A)\cdot\det(B)$ for square matrices,\\
7. $\det(A^T)=\det(A)$ for square matrices.

In other words, we propose a top-down approach to determinants: instead of building up slowly via definitions, we propose one equation that implies all of the above properties. It also leads naturally to basic concepts of Linear Algebra: linear combinations, linear independence, basis, dimension.

\end{abstract}

\section{Introduction}
Linear Algebra, a fundamental component of the undergraduate STEM curriculum, is often acknowledged as a challenging course. Its abstract concepts gradually unfold, unveiling the essence of the subject: the art of solving linear equations. In this paper, we present a top-down approach that promptly introduces the determinant, serving as a natural gateway to a comprehensive understanding of all fundamental concepts in Linear Algebra.

To foster an enriching educational experience, we have deliberately chosen not to provide exhaustive proofs and solutions. We firmly believe that an overemphasis on memorization and regurgitation detracts from genuine learning which can only be achieved by students doing their own thinking as much as possible. Consequently, we refrain from permanently posting detailed solutions to homework problems on the internet.

For a thorough comprehension of the concepts expounded in our paper, we recommend consulting the referenced works \cite{GS1} or \cite{GS2}.

The author wishes to express sincere gratitude to Dr.Nikolay Brodskiy for his invaluable assistance in organizing the GTA Mentoring program in Linear Algebra. This program played a pivotal role in significantly deepening the author's understanding of this mathematical field and refining his teaching methodologies. Special thanks are also extended to the author's former PhD students who actively participated in the program: Dr.Kyle Austin, Dr.Michael Holloway, Dr.Ryan Jensen, Dr.Kevin Sinclair, Dr.Logan Higginbotham, Dr.Pawe\l\ Grzegrz\' o\l ka, Dr.Thomas Weighill, and Dr.Jeremy Siegert.

\section{Main equation}\label{MainEquation}

Given a vector space $V$ over a field $\mathbb{F}$ we are interested in non-zero functions $D:V^n\to \mathbb{F}$ such that for any $(n+1)$-tuple of vectors $v_1,\ldots,v_n,b$ in $V$ the following equation holds
$$D(v_1,\ldots,v_n)\cdot b=D(b,v_2,\ldots,)\cdot v_1+D(v_1,b,\ldots)\cdot v_2+\ldots+D(v_1,\ldots,b)\cdot v_n,$$
where the $k$-th element of the sum is obtained by replacing $v_k$ with $b$.
Notice it implies existence of an $n$-tuple $v_1,\ldots,v_n$ in $V$ such that every element of $V$ is their linear combination. 

\begin{Example}\label{Example1}
Consider $D:\mathbb{R}\times \mathbb{R}\to \mathbb{R}$ defined as follows:\\
1. $D(x,y)=0$ if $x=0$ or $y=0$,\\
2. $D(x,y)=x-y$ if $x\ne 0$ and $y\ne 0$.

It is easy to check that $D$ does satisfy
$$D(v_1,v_2)\cdot b=D(b,v_2)\cdot v_1+D(v_1,b)\cdot v_2$$
for any $3$-tuple $v_1,v_2,b$ in $\mathbb{R}$.
\end{Example}

To avoid the situation from Example \ref{Example1} we propose the following

\begin{Definition}\label{MainDefinition}

Given a vector space $V$ over a field $\mathbb{F}$, an \textbf{$n$-determinant} is a non-zero function $D:V^n\to \mathbb{F}$ such that for any $(n+1)$-tuple of vectors $v_1,\ldots,v_n,b$ in $V$ the following equation holds
$$D(v_1,\ldots,v_n)\cdot b=D(b,v_2,\ldots,)\cdot v_1+D(v_1,b,\ldots)\cdot v_2+\ldots+D(v_1,\ldots,b)\cdot v_n,$$
where the $k$-th element of the sum is obtained by replacing $v_k$ with $b$
and $V$ is not spanned by any set of $(n-1)$ vectors in $V$. In other words, the \textbf{dimension} of $V$ over $\mathbb{F}$ is $n$.
\end{Definition}

\begin{Observation}\label{1dimCase}
Any $1$-determinant $D$ on $\mathbb{F}$ is of the form
$$D(x)=m\cdot x$$
for some $m\in \mathbb{F}$, $m\ne 0$.
\end{Observation}
\begin{proof}
Left to the reader.
\end{proof}

\begin{Observation}\label{Nullity}
If $n > 1$ and $D$ is an $n$-determinant on $V$, then
$D(v_1,\ldots,v_n)=0$ if any two vectors $v_k$ and $v_m$, for some $k\ne m$, are parallel.
\end{Observation}

Indeed, otherwise the dimension of $V$ over $\mathbb{F}$ is less than $n$. More generally,

\begin{Proposition}\label{LinearDependenceProp}
If $n > 1$ and $D$ is an $n$-determinant on $V$, then
$D(v_1,\ldots,v_n)=0$ if vectors $ v_1,\ldots,v_n$ are linearly dependent.
\end{Proposition}
\begin{proof}
If vectors $ v_1,\ldots,v_n$ are linearly dependent and $D(v_1,\ldots,v_n)\ne 0$, then a proper subset of $ v_1,\ldots,v_n$ spans $V$, a contradiction.
\end{proof}

A \textbf{multilinear function} is a generalization of a linear function that operates on multiple vector spaces simultaneously. Unlike a bilinear function that is linear in each variable separately, a multilinear function is linear in each variable independently while allowing all variables to vary.

Formally, let $V_1$, $V_2$, \ldots , $V_n$ be vector spaces over the same field $\mathbb{F}$. A multilinear function, denoted as $f: V_1\times V_2\times\ldots V_n\to V$, where $V$ is a vector space over $\mathbb{F}$, satisfies the following properties:

\textbf{Additivity}: The function $f$ is \textbf{additive} in each variable independently which means that for any two vectors $u_k, v_k\in V_i $ for some $k\leq n$, the function $f$ satisfies:
$$f(u_1, u_2, \ldots, u_k+v_k, \ldots, u_n) = f(u_1, u_2,\ldots, u_k, \ldots, u_n) +  f(u_1, u_2,\ldots, v_k, \ldots, u_n).$$

\textbf{Homogeneity}: The function $f$ is \textbf{homogeneous} of degree $1$ in each variable independently which means that for any vectors $u_k\in V_k$, where $k\leq n$ is given,  and for any scalar $c\in \mathbb{F}$, the function $f$ satisfies:
$$f(u_1, u_2, \ldots, c\cdot u_k, \ldots, u_n) = c\cdot f(u_1, u_2,\ldots, u_k, \ldots, u_n).$$

In other words, a multilinear function is linear in each variable independently, allowing all variables to vary while preserving additivity and homogeneity properties.

\begin{Proposition}\label{MultiLinearity}
If $n \ge 1$ and $D$ is an $n$-determinant on $V$, then $D$ is 
a multilinear function.
\end{Proposition}
\begin{proof}
It is obvious for $n=1$ by \ref{1dimCase}, so assume $n > 1$.
Given an $n$-tuple $T=\{v_1,\ldots,v_n\}$ of vectors in $V$ and given $s\in \mathbb{F}$, $k\leq n$, assume $D(T')\ne
s\cdot D(T)$, where $T'$ is obtained from $T$ by replacing $v_k$ by $s\cdot v_k$. Multiply the main equation \ref{MainDefinition} for $b$ by $s$ and subtract from the main equation \ref{MainDefinition} for $s\cdot b$. It gives a linear combination for vectors in $T$ being $0$ with at least one coefficient being non-zero. Hence, vectors in $T$
are linearly dependent. So are vectors in $T'$ resulting in $D(T')=0$ by Proposition \ref{LinearDependenceProp}, a contradiction.

Given an $n$-tuple $T=\{v_1,\ldots,v_n\}$ of vectors in $V$ and given $w_k\in V$, for some $k\leq n$, assume $D(T'')\ne
D(T)+D(T')$, where $T'$ is obtained from $T$ by replacing $v_k$ by $w_k$
and $T''$ is obtained from $T$ by replacing $v_k$ by $v_k+w_k$. Given $b\in V$, add the main equations for $T$ and $T'$ and subtract from the main equation for $T''$. It gives $b$ as a linear combination for vectors in $T\setminus \{v_k\}$, a contradiction.
\end{proof}

An \textbf{antisymmetric function}, also known as an \textbf{alternating function}, is a special type of multivariable function that exhibits a specific symmetry property. In particular, an antisymmetric function changes sign when any pair of its arguments is interchanged.

Formally, let $V_1$, $V_2$, \ldots , $V_n$ be vector spaces over the same field $\mathbb{F}$. A function $f: V_1\times V_2\times\ldots V_n\to V$, where $V$ is a vector space over $\mathbb{F} $is \textbf{antisymmetric} if it satisfies the following property:
$$f(u_1, u_2, \ldots, u_j, \ldots u_i, \ldots, u_n) = - f(u_1, u_2,\ldots, u_i, \ldots, u_j,\ldots, u_n)$$
 for any $i, j\leq n$, where $1 \leq i < j \leq n$. In other words, if we interchange any pair of arguments, the value of the antisymmetric function changes sign.

\begin{Proposition}\label{AntiSymmetry}
If $n > 1$ and $D$ is an $n$-determinant on $V$, then $D$ is 
an antisymmetric function.
\end{Proposition}
\begin{proof}
Given an $n$-tuple $n$-tuple $T=\{v_1,\ldots,v_n\}$ of vectors in $V$ and given $k < m\leq n$, consider the $n$-tuple $T'$ in which both $v_m$ and $v_m$ are replaced by $v_k+v_m$. $D(T')=0$ by \ref{Nullity},
apply $D$ being multilinear by \ref{MultiLinearity} to conclude $D$ is antisymmetric.
\end{proof}

\begin{Proposition}\label{LinearInDependenceProp}
If $n > 1$ and $D$ is an $n$-determinant on $V$, then
$D(v_1,\ldots,v_n)\ne 0$ if and only if vectors $ v_1,\ldots,v_n$ are linearly independent.
\end{Proposition}
\begin{proof}
If $D(v_1,\ldots,v_n)\ne 0$, then vectors $ v_1,\ldots,v_n$ span $V$, so they are are linearly independent. If $D(v_1,\ldots,v_n)=0$ and vectors $v_1,\ldots,v_n$ are linearly independent, then they span $V$
and $D(w_1,\ldots,w_n)=0$ for any other $n$-tuple $w_1,\ldots,w_n$ of vectors in $V$ by \ref{MultiLinearity}, \ref{Nullity}, and \ref{AntiSymmetry}, a contradiction.
\end{proof}

Let us generalize row operations on matrices to \textbf{elementary tuple operations} on $n$-tuples of vectors in $V$. Some of them we already used.

There are three types of elementary tuple operations:\\
\textbf{Element Interchange}: Swapping two elements of a tuple.\\
\textbf{Element Scaling}: Multiplying an element of a tuple by a nonzero scalar.\\
\textbf{Element Replacement}: Adding a multiple of one element to another element.

\begin{Theorem}\label{CharOfDeterminants}
If $D:V^n\to \mathbb{F}$ is a multilinear and antisymmetric function, then
for any $(n+1)$-tuple $T=\{v_1,\ldots,v_n,b\}$ of vectors in $V$ the equation
$$D(v_1,\ldots,v_n)\cdot b=D(b,v_2,\ldots,)\cdot v_1+D(v_1,b,\ldots)\cdot v_2+\ldots+D(v_1,\ldots,b)\cdot v_n$$
holds.
\end{Theorem}
\begin{proof}
Define 
$E(T):= D(v_1,\ldots,v_n)\cdot b-D(b,v_2,\ldots,)\cdot v_1+D(v_1,b,\ldots)\cdot v_2+\ldots+D(v_1,\ldots,b)\cdot v_n$. If $T'$ is obtained from $T$ by an elementary tuple operation, then the nullity of $E(T)$ is equivalent to the nullity of $E(T')$.
By applying tuple operations we reduce the general case to that of $T$ consisting of elements of a given basis of $V$. In that case either one of the vectors is $0$, so $E(T)=0$ or $b=e_k$ for some $k$, so $E(T)=0$ again.
\end{proof}

\begin{Corollary}[Cramer's Rule]\label{CramersRule}
If $D$ is an $n$-determinant on a vector space $V$ over a field $\mathbb{F}$
and $v_1,\ldots,v_n$ are linearly independent vectors in $V$, then any linear
equation
$$x_1\cdot v_1+\ldots+x_n\cdot v_n=b$$
has unique solutions given by
$$x_k=\frac{D(v_1,\ldots,b,\ldots,v_n)}{D(v_1,\ldots,v_n)}.$$
\end{Corollary}
\begin{proof}
Left to the reader.
\end{proof}

\begin{Corollary}[Uniqueness of determinants]\label{UniquenessOfDet}
If $D_1$ is an $n$-determinant on a vector space $V$ over a field $\mathbb{F}$ and $D_2:V^n\to \mathbb{F}$ is a multilinear and antisymmetric function, then there is $c\in  \mathbb{F}$ such that
$$D_2(T)=c\cdot D_1(T)$$
for any $n$-tuple $T$ of vectors in $V$. In particular, if $D_2$ is an $n$-determinant, then $c\ne 0$.
\end{Corollary}
\begin{proof}
Pick an $n$-tuple $T$ in $V$ so that
$D_1(T)\ne 0$ and put $c=\frac{D_2(T)}{D_1(T)}$.
Define $D$ as $D_2-c\cdot D_1$. $D$ is a multilinear and antisymmetric function. It cannot be an $n$-determinant as $D(T)=0$ (use \ref{LinearInDependenceProp}), so $D$ must be equal to $0$.
\end{proof}

\section{Creating determinants}\label{CreatingDeterminants}
In this section we show existence of determinants on vector spaces of finite dimension.

\begin{Proposition}
If $D$ is an $n$-determinant on $V$ and $W$ is a subspace of $V$, then there is a $k$-determinant $D'$ on $W$.
\end{Proposition}
\begin{proof}
Choose a basis $e_1,\ldots, e_k$ of $W$ and extend it to a basis
$e_1,\ldots,e_n$ of $V$. Given vectors $w_1,\ldots,w_k$ of $W$ define
$$D'(w_1,\ldots,w_k):=D(w_1,\ldots,w_k,e_{k+1},\ldots,e_n).$$
Notice $D'$ is non-zero, multilinear, and antisymmetric.
\end{proof}

In the next result we use one more operation on tuples of vectors, namely dropping one of its elements. That position in the tuple is indicated by the sign $\hat{}$ over it.
\begin{Proposition}\label{ExtendingDeterminants}
If $D$ is an $n$-determinant on $V$ and $W=\mathbb{F}\times V$, then
$D'$ is an $(n+1)$-determinant on $W$, where
$D'(w_1,\ldots,w_{n+1})$ is defined as the sum
$$\sum\limits_{i=1}^{i=n+1}(-1)^{i-1}\cdot p_1(w_i)\cdot D(p_2(w_1),\ldots,\widehat{p_2(w_i)}\ldots),$$
where $p_1:\mathbb{F}\times V\to \mathbb{F}$ is the projection onto the first coordinate
and $p_2:\mathbb{F}\times V\to V$ is the projection onto the second coordinate.
\end{Proposition}
\begin{proof}
$D'$ is non-zero as follows: extend a basis $\{v_2,\ldots,v_{n+1}\}$ in $V$ to that of $W$ by declaring $w_1=(1,0)$ and $w_i=(0,v_i)$ for $i > 1$.
Notice $D'(w_1,\ldots,w_{n+1})=D(v_2,\ldots,v_{n+1})\ne 0$.

Clearly $D'$ is multilinear. To show it is antisymmetric it is sufficient to swap two adjacent vectors $w_k$ and $w_{k+1}$. It is left to the reader to complete the proof.
\end{proof}

\begin{Example}
Extending $D(x)=x$ on $\mathbb{F}$ to $\mathbb{F}\times\mathbb{F}$ as in \ref{ExtendingDeterminants} yields
$$\det((a_{11},a_{12}),(a_{21},a_{22})):=a_{11}\cdot a_{22}-a_{21}\cdot a_{12}$$
as the unique $2$-determinant on $\mathbb{F}\times\mathbb{F}$ satisfying
$$\det((1,0),(0,1))=1.$$
\end{Example}

Notice that \ref{Example1} indicates that dropping the assumption of $V$ being $n$-dimensional leads to non-uniqueness of functions $D$ satisfying the main equation \ref{MainDefinition}, a situation which is needed to be avoided.

\section{Determinants of square matrices}

By applying \ref{ExtendingDeterminants} inductively we see that for any
$n\ge 1$ the vector space $V=\mathbb{F}^n$ has a unique determinant $\det$
such that $\det(e_1,\ldots,e_n)=1$. In particular, $\dim(\mathbb{F}^n)=n$.
Since each square $n\times n$ matrix $A$ belongs to $V^n$, where each row of $A$ is considered to be in $V$, we have a natural concept of the determinant of $A$. The aim of this section is to derive easily the basic properties of determinants of square matrices.

\begin{Observation}\label{DiagonalObservation}
$\det(A)$ is the product of elements on the diagonal of $A$ if $A$ is a diagonal matrix.
\end{Observation}

Given an $n\times n$-matrix $A=\{a_{i,j}\}$ over a field $\mathbb{F}$
and given an $n$-tuple $T=\{v_1,\ldots,v_n\}$ of vectors in a vector space $V$ over $\mathbb{F}$, define $A\cdot T$ as the $n$-tuple $\{w_1,\ldots,w_n\}$, where $w_i=\sum\limits_{j=1}^n a_{i,j}\cdot v_j$.

\begin{Observation}\label{DOfATObservation}
Given an $n$-determinant $D$ on $V^n$ and given a square $n\times n$ matrix $A$ over $F$, a sequence of elementary row operations leading from $A$ to a diagonal matrix $C$ leads to analogous tuple operations on $A\cdot T$ leading to $C\cdot T$, hence
$$D(A\cdot T)=\det(A)\cdot D(T).$$
\end{Observation}

\begin{Exercise}
Given two $n\times n$-matrices $A$ and $B$ and given an $n$-tuple $T$ of vectors in $V$ one has
$$(A\cdot B)\cdot T=A\cdot (B\cdot T).$$
\end{Exercise}

\begin{Corollary}
$\det(A\cdot B)=\det(A)\cdot \det(B)$.
\end{Corollary}
\begin{proof}
Pick $T=I$, the identity matrix. By \ref{DOfATObservation},
$\det(A\cdot B)=\det((A\cdot B)\cdot I)=\det(A\cdot (B\cdot I))=
\det(A)\cdot \det (B\cdot I)=\det(A)\cdot \det(B)$.
\end{proof}

\textbf{Elementary $n\times n$ matrices} $E$ over a field $\mathbb{F}$ are square matrices that represent elementary tuple operations performed on $n$-tuples $T$ of vectors in $V$. Namely, $E\cdot T$ amounts to an elementary tuple operation on $T$.

An \textbf{elementary matrix} is created by applying a single elementary row operation to an identity matrix of the same size. The resulting matrix represents the effect of analogous elementary tuple operation on an $n$-tuple when multiplied by it on the left.

\begin{Exercise}
For each type of elementary tuple operation, there is a corresponding elementary matrix:\\
\textbf{Elementary Matrix for Element Interchange}: It is obtained by swapping the corresponding rows of the identity matrix.\\
\textbf{Elementary Matrix for Element Scaling}: It is obtained by multiplying a row of the identity matrix by a non-zero scalar.\\
\textbf{Elementary Matrix for Element Replacement}: It is obtained by adding a multiple of one row to another row in the identity matrix.
\end{Exercise}

\begin{Observation}\label{FactorizationOfMatrices}
Since each square matrix $A$ over $F$ can be reduced to a diagonal matrix $D$ by applying the three elementary row operations and each elementary matrix has the inverse being a matrix of the same type, $A$ can be represented as
the product $E_1\cdot \ldots \cdot E_k\cdot C$ of elementary matrices $E_i$, $i\leq k$, and of a diagonal matrix $C$.
\end{Observation}

\begin{Corollary}
$\det(A)$ is the product of entries of $C$ and $(-1)^m$, where $m$
is the number of $E_i$ that swap rows.
\end{Corollary}

\begin{Corollary}
$\det(A^T)=\det(A)$.
\end{Corollary}
\begin{proof}
Use $(M\cdot N)^T=N^T\cdot M^T$ and the fact $\det(E^T)=\det(E)$
for elementary matrices plus the analogous fact for diagonal matrices.
\end{proof}

\textbf{Cofactor expansion}, also known as expansion by minors or expansion by cofactors, is a method used to calculate the determinant of a square matrix. 

Cofactor expansion involves expanding the determinant along a row or a column of the matrix. Let's consider a square matrix $A$ of size $n \times n$. To calculate the determinant of A using cofactor expansion along the $i$-th row, the formula is as follows:

$det(A) = a_{i1} C_{i1} + a_{i2} C_{i2} + ... + a_{in} C_{in}$

Here, $a_{ij}$ represents the element of the matrix $A$ in the $i$-th row and $j$-th column, and $C_{ij}$ represents the cofactor associated with that element. The cofactor $C_{ij}$ is defined as the determinant of the submatrix obtained by removing the $i$-th row and $j$-th column from the original matrix $A$, multiplied by $(-1)^{i+j}$.

Cofactor expansion can also be performed along a column instead of a row. The formula remains the same, but the expansion is done along the $j$-th column, and the cofactors are defined accordingly.

\begin{Corollary}[Cofactor Expansion]\label{CofactorExpansion}
Cofactor expansion is a valid method of reducing determinants of $n\times n$-matrices to a combination of determinants of $(n-1)\times (n-1)$-matrices.
\end{Corollary}
\begin{proof}
Notice that \ref{ExtendingDeterminants} is precisely cofactor expansion along the first column of $A$. Apply column operations to deduce all cases of cofactor expansion.

Alternatively, create formulae analogous to \ref{ExtendingDeterminants} for each case of cofactor expansion to see that it does represent the standard determinant.
\end{proof}

\end{document}